\documentclass{article}

\usepackage{
	amsmath, 
	amsthm, 
	amssymb,
	fancyhdr, 
	setspace,
	tikz-cd,
}
\usepackage{upgreek}

\usepackage{lscape}
\usepackage{color}
\usepackage[color,matrix,arrow]{xy}
\usepackage{setspace}
\xyoption{all}

\DeclareMathOperator{\pd}{pd}
\DeclareMathOperator{\id}{id}

\usepackage{avant}
\usepackage{amsmath, amssymb, amsfonts, color, tikz, bbm, txfonts, euscript}
\usepackage{amssymb}
\usepackage{enumerate}
\usepackage{bm}
\newtheorem{theorem}{Theorem}[section]

\newtheorem{prop}[theorem]{Proposition}
\newtheorem{cor}[theorem]{Corollary}
\theoremstyle{definition}

\newtheorem{remark}[theorem]{Remark}
\begin{document}

\thispagestyle{empty}

\title{Several Results Concerning Convex Subcategories}
\author{Stephen Zito}
        
\maketitle

\begin{abstract}
We apply the notion of a full convex subcategory to a wide range of algebras including tilted, quasi-tilted, shod, weakly shod, left and right glued, laura, simply connected, strongly simply connected, left supported, and cluster-tilted.  In particular, given an algebra $\Lambda$ from one of the aforementioned classes, we investigate certain factor algebras $\Lambda/I$ where $I$ is an ideal generated by a suitable idempotent.   
\end{abstract}

\section{Introduction} 
Given an algebra $\Lambda$ and an ideal $I$ in $\Lambda$, we can ask what properties of $\Lambda$ are inherited by the factor algebra $\Lambda/I$.  In general, this is a difficult problem.  However, if we consider ideals generated by idempotents, that is, $I=<e>$ for some idempotent $e\in\Lambda$, we can impose certain conditions and deduce useful information.  Consider an algebra of the form $\Lambda=KQ/I$ where $K$ is an algebraically closed field, $Q$ is a finite quiver, $KQ$ the path algebra, and $I$ an ideal.  Let $C$ be a full convex subquiver of $Q$, that is, if $p$ is a path from a vertex in $C$ to a vertex in $C$, then each arrow and vertex of $p$ is in $C$.  Let $e_C$ be the idempotent associated to the sum of the vertices in $C$ and $e'_C$ be the idempotent associated to the sum of the vertices not in $C$.  We wish to examine $\Lambda/<e'_C>$.      
 \par
 We start by considering several classes of algebras that have been extensively studied in the representation theory of finite dimensional algebras, namely the tilted algebras $\cite{HR}$, quasi-tilted algebras $\cite{HRS}$, the shod algebras $\cite{CL,RS}$, the weakly shod algebras $\cite{CL2}$, the left and right glued algebras $\cite{AC}$, and finally, the laura algebras $\cite{AC2,RS2}$.  Our first main result says, if $\Lambda$ belongs to one of these classes, so does $\Lambda/<e'_C>$.
\begin{theorem}$\emph{[Theorem~\ref{main1}]}$ 
Let $\Lambda=KQ/I$ and $C$ a full convex subquiver of $Q$.  Let $e_C$ be the idempotent associated to the sum of the vertices in $C$ and $e'_C$ be the idempotent associated to the sum of the vertices not in $C$.
\begin{itemize}
\item[\emph{(a)}] If $\Lambda$ is laura, so is $\Lambda/<e'_C>$.
\item[\emph{(b)}] If $\Lambda$ is left (or right) glued, so is $\Lambda/<e'_C>$.
\item[\emph{(c)}] If $\Lambda$ is weakly shod, so is $\Lambda/<e'_C>$.
\item[\emph{(d)}] If $\Lambda$ is shod, so is $\Lambda/<e'_C>$.
\item[\emph{(e)}] If $\Lambda$ is quasi-tilted, so is $\Lambda/<e'_C>$.
\item[\emph{(f)}] If $\Lambda$ is tilted, so is $\Lambda/<e'_C>$.
\end{itemize}
\end{theorem}
For our second main result, we consider simply connected and strongly simply connected algebras $\cite{AS,S}$.
\begin{theorem}$\emph{[Theorem~\ref{main2}]}$ 
\label{simple}
Let $\Lambda=KQ/I$ and $C$ a full convex subquiver of $Q$.  Let $e_C$ be the idempotent associated to the sum of the vertices in $C$ and $e'_C$ be the idempotent associated to the sum of the vertices not in $C$.  If $\Lambda$ is strongly simply connected, then $\Lambda/<e'_C>$ is simply connected.
\end{theorem}
We have the following corollary when $\Lambda$ is representation-finite.
\begin{cor}$\emph{[Corollary~\ref{cor1}]}$ 
Let $\Lambda=KQ/I$ be representation finite and $C$ a full convex subquiver of $Q$.  Let $e_C$ be the idempotent associated to the sum of the vertices in $C$ and $e'_C$ be the idempotent associated to the sum of the vertices not in $C$. 
\begin{itemize}
\item[\emph{(a)}] If $\Lambda$ is strongly simply connected, then $\Lambda/<e'_C>$ is strongly simply connected.
\item[\emph{(b)}] If $\Lambda$ is simply connected, then $\Lambda/<e'_C>$ is strongly simply connected.
\end{itemize}
\end{cor}

Our third main result deals with the left support and left supported algebras $\cite{ACT}$.
\begin{theorem}$\emph{[Theorem~\ref{main3}]}$ 
\label{left1}
Let $\Lambda=KQ/I$ be an algebra and $\Lambda_{\lambda}$ the left support. Let $e'$ be the direct sum of the idempotents not in $\Lambda_{\lambda}$.
\begin{itemize}
\item[\emph{(a)}] $\Lambda/<e'>$ is a direct product of quasi-tilted algebras.
\item[\emph{(b)}] If $\Lambda$ is left supported, then $\Lambda/<e'>$ is a direct product of tilted algebras.
\end{itemize}
\end{theorem}
When $\Lambda$ is a cluster-tilted algebra $\cite{BMR1}$, we can be more precise.
\begin{cor}$\emph{[Corollary~\ref{DS}]}$

Let $\Lambda=KQ/I$ be a cluster-tilted algebra and $\Lambda_{\lambda}$ the left support.  Let $e'$ be the direct sum of the idempotents not in $\Lambda_{\lambda}$.  Then $\Lambda/<e'>$ is a direct product of hereditary algebras.
\end{cor}
Let $\Lambda$ be cluster-tilted with $e\in\Lambda$ an idempotent.  Then the full subcategory $e\Lambda e$ is not necessarily cluster-tilted.  However, this is the case when $e\Lambda e$ is convex.
\begin{prop}$\emph{[Proposition~\ref{cluster}]}$ 

Let $\Lambda=KQ/I$ be a cluster-tilted algebra with $C$ a full convex subquiver of $Q$.  Let $e_C$ be the idempotent associated to the sum of the vertices in $C$.  Then $e_C\Lambda e_C$ is cluster-tilted.
\end{prop}

\section{Notation And Preliminaries}
We set the notation for the remainder of this paper. All algebras are assumed to be finite dimensional over an algebraically closed field $K$.  If $\Lambda$ is a $K$-algebra then denote by $\mathop{\text{mod}}\Lambda$ the category of finitely generated right $\Lambda$-modules and by $\mathop{\text{ind}}\Lambda$ a set of representatives of each isomorphism class of indecomposable right $\Lambda$-modules.  Given $M\in\mathop{\text{mod}}\Lambda$, the projective dimension of $M$ is denoted $\pd_{\Lambda}M$ and the injective dimension by $\id_{\Lambda}M$.  We let $\mathop{\text{gl.dim}}\Lambda$ stand for the global dimension of an algebra $\Lambda$
\par
Let $\Lambda=KQ/I$ where $K$ is an algebraically closed field, $Q$ is a finite quiver, $KQ$ the path algebra, and $I$ an ideal.  We say a subquiver $C$ is $\it{full}$ if for any two vertices $v$ and $w$ of $C$, all the arrows in $Q$ with origin $v$ and terminus $w$ are also arrows in $C$.  We say a full subquiver $C$ of $Q$ is $\it{convex}$ if for any two vertices $v,w$ in $C$ and for any path $p$ from $v$ to $w$, every vertex occurring in $p$ is in $C$.  If $C$ is a full subquiver of $Q$, let $e_C$ be the idempotent associated to the sum of the vertices in $C$ and $e'_C$ be the idempotent associated to the sum of the vertices not in $C$.  Then $e_C$ and $e'_C$ are orthogonal idempotents in $KQ$ and $1=e_C+e'_C$.  We abuse notation and view $e_C$ and $e'_C$ as idempotents in both $KQ$ and $\Lambda$.  In general, $\Lambda/<e'_C>$ and $e_C\Lambda e_C$ are not isomorphic.  When $C$ is convex, we do in fact have an isomorphism.
\begin{prop} $\emph{\cite[Proposition~3.3(d)]{GM}}$ 
\label{crux}
$C$ be a full subquiver of $Q$ and $\Lambda=KQ/I$.  Let $e_C$ be the idempotent associated to the sum of the vertices in $C$ and $e'_C$ be the idempotent associated to the sum of the vertices not in $C$.  If $C$ is convex, then $\Lambda/<e'_C>$ is isomorphic to $e_C\Lambda e_C$.
\end{prop}

\section{Main Results}
\subsection{Full Subcategories}
Since its introduction by Happel and Ringel in the early eighties, $\cite{HR}$, the class of tilted algebras have been extensively studied in the representation theory of finite dimensional algebras.  It was natural to consider generalizations of this notion.  Thus, over the years, the following classes of algebras were defined and studied: the quasi-tilted algebras $\cite{HRS}$, the shod algebras $\cite{CL,RS}$, the weakly shod algebras $\cite{CL2}$, the left and right glued algebras $\cite{AC}$, and the laura algebras $\cite{AC2,RS2}$.  Let $\Lambda$ belong to one of the aforementioned classes.  A nice feature of these classes of algebras is, given an idempotent $e\in\Lambda$, then the endomorphism algebra $e\Lambda e$ also belongs to the same class.
\begin{theorem}$\emph{\cite{H,AC3}}$ 
\label{subcat}
Let $\Lambda$ be an algebra and $e\in\Lambda$ an idempotent.
\begin{itemize}
\item[\emph{(a)}] If $\Lambda$ is laura, so is $e\Lambda e$.
\item[\emph{(b)}] If $\Lambda$ is left (or right) glued, so is $e\Lambda e$.
\item[\emph{(c)}] If $\Lambda$ is weakly shod, so is $e\Lambda e$.
\item[\emph{(d)}] If $\Lambda$ is shod, so is $e\Lambda e$.
\item[\emph{(e)}] If $\Lambda$ is quasi-tilted, so is $e\Lambda e$.
\item[\emph{(f)}] If $\Lambda$ is tilted, so is $e\Lambda e$.
\end{itemize}
\end{theorem}
We can combine Proposition $\ref{crux}$ and Theorem $\ref{subcat}$ to obtain our first main result.
\begin{theorem}
\label{main1}
Let $\Lambda=KQ/I$ and $C$ a full convex subquiver of $Q$.  Let $e_C$ be the idempotent associated to the sum of the vertices in $C$ and $e'_C$ be the idempotent associated to the sum of the vertices not in $C$.
\begin{itemize}
\item[\emph{(a)}] If $\Lambda$ is laura, so is $\Lambda/<e'_C>$.
\item[\emph{(b)}] If $\Lambda$ is left (or right) glued, so is $\Lambda/<e'_C>$.
\item[\emph{(c)}] If $\Lambda$ is weakly shod, so is $\Lambda/<e'_C>$.
\item[\emph{(d)}] If $\Lambda$ is shod, so is $\Lambda/<e'_C>$.
\item[\emph{(e)}] If $\Lambda$ is quasi-tilted, so is $\Lambda/<e'_C>$.
\item[\emph{(f)}] If $\Lambda$ is tilted, so is $\Lambda/<e'_C>$.
\end{itemize}
\end{theorem}
\begin{proof}
Consider $e_C\Lambda e_C$.  Theorem $\ref{subcat}$ guarantees $e_C\Lambda e_C$ will be the same class type as $\Lambda$.  Now consider $\Lambda/<e'_C>$ and apply Proposition $\ref{crux}$.
\end{proof}
\begin{remark}
Since $e_C$ corresponds to a convex sub quiver $C$ of $Q$, the global dimension of $e_C\Lambda e_C$ does not exceed that of $\Lambda$.  Let $B=e_C\Lambda e_C$.  It is well known that, in this case, for any two $B$-modules $X$, $Y$, we have $\text{Ext}_B^i(X,Y)\cong\text{Ext}_{\Lambda}^i(X,Y)$ for all $i$.  Hence, for any simple $B$-module $S$, we have $\pd_BS\leq\pd_{\Lambda}S$.  This yields the statement.  Combining with Theorem $\ref{main1}$, we have $\text{gl.dim}(\Lambda)\geq\text{gl.dim}(\Lambda/<e'_C>)$.
\end{remark}
\begin{remark}
In general, $\Lambda/<e>$ is not guaranteed to be of the same class as $\Lambda$ for an arbitrary idempotent $e\in\Lambda$.  The following example is taken from $\cite{BMR2}$.  Consider the path algebra of the following quiver 
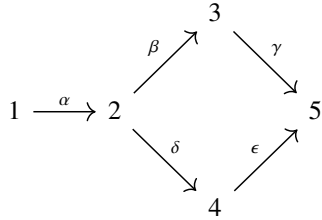
\[
\begin{tikzcd}[row sep=5ex]
&& 3\arrow[dr,"\gamma"] \\
1 \arrow[r,"\alpha"] &2 \arrow[ur,"\beta"] \arrow[dr,"\delta"]&&5 \\
&& 4\arrow[ur,"\epsilon"]
\end{tikzcd}
\]
with relations $\alpha\beta=\beta\gamma-\delta\epsilon=0$.  This is a tilted algebra.  Let $e_4$ be the primitive idempotent corresponding to vertex $4$.  Then $\Lambda/<e_4>$ is not tilted.  Notice, $e=e_1+e_2+e_3+e_5$ is not convex.
\end{remark}
\subsection{Simply And Strongly Simply Connected Algebras}
An algebra $\Lambda$ is called $\it{simply~connected}$ provided its ordinary quiver $Q$ has no oriented cycles and, for any presentation of $\Lambda$ as a bound quiver algebra, $\Lambda\cong KQ/I$, the fundamental group $\pi_1(Q,I)$ is trivial $\cite{AS}$.  It is natural to ask if simply connected algebras are closed under taking full, convex subcategories, thus allowing an application of Proposition $\ref{crux}$.  This is not the case.  Consider $\Lambda=KQ/I$, where $Q$ is the quiver 
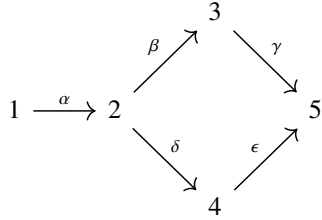
\[
\begin{tikzcd}[row sep=5ex]
&& 3\arrow[dr,"\gamma"] \\
1 \arrow[r,"\alpha"] &2 \arrow[ur,"\beta"] \arrow[dr,"\delta"]&&5 \\
&& 4\arrow[ur,"\epsilon"]
\end{tikzcd}
\]
with $I$ generated by $\alpha\beta\gamma-\alpha\delta\epsilon$.  Then $\pi_1(Q,I)=0$ for every presentation of $\Lambda$, so that $\Lambda$ is simply connected.  However, the full convex subcategory $e=e_2+e_3+e_4+e_5$ is hereditary with fundamental group $\mathbb{Z}$, thus not simply connected.  
\par
To apply Proposition $\ref{crux}$, we turn to one particular subclass of simply connected algebras.  We say an algebra $\Lambda$ is $\it{strongly~simply~connected}$ if every full, convex subcategory is simply connected $\cite{S}$.  Thus, the following result is immediate.
\begin{theorem}
\label{main2}
Let $\Lambda=KQ/I$ be a strongly simply connected algebra and $C$ a full convex subquiver of $Q$.  Let $e_C$ be the idempotent associated to the sum of the vertices in $C$ and $e'_C$ be the idempotent associated to the sum of the vertices not in $C$.  Then $\Lambda/<e'_C>$ is simply connected.
\end{theorem}
It is known that a representation-finite algebra $\Lambda$ is simply connected if and only if $\Lambda$ is strongly simply connected $\cite{BG}$.  This leads to the following corollary.
\begin{cor}
\label{cor1}
Let $\Lambda=KQ/I$ be representation finite and $C$ a full convex subquiver of $Q$.  Let $e_C$ be the idempotent associated to the sum of the vertices in $C$ and $e'_C$ be the idempotent associated to the sum of the vertices not in $C$.
\begin{itemize}
\item[\emph{(a)}] If $\Lambda$ is strongly simply connected, then $\Lambda/<e'_C>$ is strongly simply connected.
\item[\emph{(b)}] If $\Lambda$ is simply connected, then $\Lambda/<e'_C>$ is strongly simply connected.
\end{itemize}
\end{cor}
\begin{proof}
Assume $\Lambda$ is strongly simply connected.  Theorem $\ref{main2}$ says $\Lambda/<e'_C>$ is simply connected.  Since $\Lambda$ is representation-finite, so is $\Lambda/<e'_C>$.  This further implies $\Lambda/<e'_C>$ is strongly simply connected.  If $\Lambda$ is simply connected, then $\Lambda$ is strongly simply connected since we assumed $\Lambda$ is representation-finite.  The second statement now follows. 
\end{proof}
Let $\Lambda$ be an algebra of finite type and $M_1,M_2,\cdots,M_n$ be a complete set of representatives of the isomorphism classes of indecomposable $\Lambda$-modules.  Then  $A=\text{End}_{\Lambda}(\oplus_{i=1}^nM_i)$ is the $\it{Auslander~algebra}$ of $\Lambda$.  It was shown in $\cite{AB}$ that simply connected and strongly simply connected are equivalent conditions for Auslander algebras.  Thus, the following corollary is immediate.
\begin{cor}
Let $A=KQ/I$ be an Auslander algebra and $C$ a full convex subquiver of $Q$.  Let $e_C$ be the idempotent associated to the sum of the vertices in $C$ and $e'_C$ be the idempotent associated to the sum of the vertices not in $C$.
\begin{itemize}
\item[\emph{(a)}] If $A$ is strongly simply connected, then $A/<e'_C>$ is strongly simply connected.
\item[\emph{(b)}] If $A$ is simply connected, then $A/<e'_C>$ is strongly simply connected.
\end{itemize}
\end{cor}

\subsection{Left Supported And Cluster-Tilted Algebras}

Let $\Lambda$ be an algebra.  Given $X,Y\in\mathop{\text{ind}}\Lambda$, we denote $X\leadsto Y$ in case there exists a chain of nonzero nonisomorphisms
\[ 
X=X_0\xrightarrow{f_1}X_1\xrightarrow{f_2}\cdots X_{t-1}\xrightarrow{f_t} X_t=Y
\]
with $t\geq0$, between indecomposable modules.  In this case we say $X$ is a predecessor of $Y$ and $Y$ is a successor of $X$.  We denote by $\mathcal{L}_\Lambda$ the following subcategory of $\mathop{\text{ind}}\Lambda$:
\[
\mathcal{L}_{\Lambda}=\{Y\in\mathop{\text{ind}}\Lambda:\pd_{\Lambda}X\leq1~\text{for each}~X\leadsto Y\}.
\]
We call $\mathcal{L}_{\Lambda}$ the $\it{left~part}$ of the module category $\mathop{\text{mod}}\Lambda$.
It is easy to see that $\mathcal{L}_{\Lambda}$ is closed under predecessors.  The $\it{left~support}$ of $\Lambda$, $\Lambda_{\lambda}$, is the endomorphism algebra of the direct sum of all indecomposable projective $\Lambda$-modules lying in $\mathcal{L}_{\Lambda}$.  It is clear from the definition that $\Lambda_{\lambda}$ is a full convex subcategory of $\Lambda$.  It was shown in $\cite{ACT}$ that $\Lambda_{\lambda}$ is a direct product of quasi-tilted algebras.
\par    
We say an algebra $\Lambda$ is $\it{left~supported}$ provided the class $\mathop{\text{add}}\mathcal{L}_{\Lambda}$ is contravariantly finite in $\mathop{\text{mod}}\Lambda$ $\cite{ACT}$.  If $\Lambda$ is left supported, it was also shown in $\cite{ACT}$ that $\Lambda_{\lambda}$ is a direct product of tilted algebras.
\begin{theorem}  
\label{main3}
Let $\Lambda=KQ/I$ be an algebra and $\Lambda_{\lambda}$ the left support. Let $e'$ be the direct sum of the idempotents not in $\Lambda_{\lambda}$.
\begin{itemize}
\item[\emph{(a)}] $\Lambda/<e'>$ is a direct product of quasi-tilted algebras.
\item[\emph{(b)}] If $\Lambda$ is left supported, then $\Lambda/<e'>$ is a direct product of tilted algebras.
\end{itemize}
\end{theorem}     
\begin{proof}
Since $\Lambda_{\lambda}$ is a full convex subcategory of $\Lambda$, we may apply Proposition $\ref{crux}$ and the previous discussion to obtain our result.
\end{proof}
Cluster-tilted algebras are the endomorphism algebras of cluster-tilting objects over cluster categories of hereditary algebras, introduced by Buan, Marsh and Reiten in $\cite{BMR1}$.  A striking property shown in $\cite{BMR2}$ is that factoring out the two-sided ideal generated by an idempotent results in another cluster-tilted algebra.  Thus, we have the following corollary.
\begin{cor}
\label{DS}
Let $\Lambda=KQ/I$ be cluster-tilted and $\Lambda_{\lambda}$ the left support.  Let $e'$ be the direct sum of the idempotents not in $\Lambda_{\lambda}$.  Then $\Lambda/<e'>$ is a direct product of hereditary algebras.
\end{cor}
\begin{proof}
By $\cite{BMR2}$, we know $\Lambda/<e'>$ is cluster-tilted.  We also know from Theorem $\ref{main3}$ (a) that $\Lambda/<e'>$ is a direct product of quasi-tilted algebras.  It is well known that cluster-tilted algebras have global dimension equal to $1$ or $\infty$.  Since the global dimension of any quasi-tilted algebra is less than or equal to $2$, we must have $\Lambda/<e'>$ is a direct product of hereditary algebras.
\end{proof}
Corollary $\ref{DS}$ gives us a new proof of a result due to D. Smith  which says the left support of a cluster-tilted algebra is a direct product of hereditary algebras $\cite{SM}$.
\begin{prop}
\label{smith}
Let $\Lambda=KQ/I$ be cluster-tilted and $\Lambda_{\lambda}$ the left support.  Then $\Lambda_{\lambda}$ is a direct product of hereditary algebras. 
\end{prop}
\begin{proof}
Let $e'$ be the direct sum of the idempotents not in $\Lambda_{\lambda}$.  Since $\Lambda_{\lambda}$ is a full convex subcategory of $\Lambda$, Proposition $\ref{crux}$ implies $\Lambda_{\lambda}\cong\Lambda/<e'>$.  Our result follows from Corollary $\ref{DS}$.
\end{proof}

\begin{remark}
Given an algebra $\Lambda$, we have the dual notion to $\mathcal{L}_{\Lambda}$.  Denote by $\mathcal{R}_{\Lambda}$ the following subcategory of $\mathop{\text{ind}}\Lambda$
\[
\mathcal{R}_{\Lambda}=\{Y\in\mathop{\text{ind}}\Lambda:\id_{\Lambda}X\leq1~\text{for each}~Y\leadsto X\}.
\]
We call $\mathcal{R}_{\Lambda}$ the $\it{right~part}$ of the module category $\mathop{\text{mod}}\Lambda$.  An algebra $\Lambda$ is $\it{right~supported}$ if $\mathop{\text{add}}\mathcal{R}_{\Lambda}$ is covariantly finite.  This dual notion is mentioned in $\cite{ACT}$.  We leave the appropriate reformulations of Theorem $\ref{main3}$, Corollary $\ref{DS}$, and Proposition $\ref{smith}$ to the reader.
\end{remark}
We end this paper with another application to cluster-tilted algebras.  Unlike the class of algebras from Theorem $\ref{main1}$, full subcategories of cluster-tilted algebras are not necessarily cluster-tilted.  The following example is from $\cite{A}$.  Consider the cluster-tilted algebra $\Lambda$ given by the following quiver 
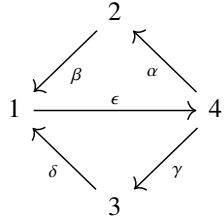
\[
\begin{tikzcd}[row sep=5ex]
&2\arrow[dl,"\beta"]\\
1 \arrow[rr,"\epsilon"] &&4 \arrow[dl,"\gamma"]\arrow[ul,"\alpha"]\\
&3\arrow[ul,"\delta"]
\end{tikzcd}
\]
with relations $\alpha\beta=\gamma\delta$, $\epsilon\alpha=0$, $\epsilon\gamma=0$, $\delta\epsilon=0$, and $\beta\epsilon=0$.  Let $e=e_1+e_4$, then $e\Lambda e$ is given by the quiver
 
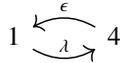
\[
\begin{tikzcd}[row sep=5ex]
1 \arrow[r,bend right,"\lambda"]&4\arrow[l,bend right, swap,"\epsilon"]
\end{tikzcd}
\]
with relations $\epsilon\lambda=0$, $\lambda\epsilon=0$.  This is not a cluster-tilted algebra because its quiver contains a $2$-cycle.  Instead, if we focus on convex subcategories, we obtain the desired result.
\begin{prop}
\label{cluster}
Let $\Lambda=KQ/I$ be a cluster-tilted algebra with $C$ a full convex subquiver of $Q$.  Let $e_C$ be the idempotent associated to the sum of the vertices in $C$.  Then $e_C\Lambda e_C$ is cluster-tilted.
\end{prop}
\begin{proof}
Let $e'_C$ be the idempotent associated to the sum of the vertices not in $C$.  By $\cite{BMR2}$, we know $\Lambda/<e'_C>$ is cluster-tilted.  Proposition $\ref{crux}$ gives an isomorphism $e_C\Lambda e_C\cong\Lambda/<e'_C>$ and our result follows.
\end{proof}

\noindent Department of Mathematics, University of Connecticut-Waterbury, Waterbury, CT 06702, USA
\\
\it{E-mail address}: \bf{stephen.zito@uconn.edu}


\begin{thebibliography}{1}


\bibitem{A} I. Assem, A course on cluster-tilted algebras, In: $\emph{Homological~Methods}$, $\emph{Representation~Theory,~and~Cluster~Algebras}$, Springer, Berlin Heidelberg, New York, (2018), 127---176.


\bibitem{AB} I. Assem and P. Brown, Strongly simply connected Auslander algebras, $\emph{Glasgow~Math.~J.}~\bf{39}$ (1997), 21--27.


\bibitem{AC} I. Assem and F. Coelho, Glueings of tilted algebras, $\emph{J.~Pure~Applied~Algebra}~\bf{96}$ (1994), no. 3, 225--243.

\bibitem{AC2} I. Assem and F. Coelho, Two-sided gluings of tilted algebras, $\emph{J.~Algebra}~\bf{269}$ (2003), no. 2, 456--479.

\bibitem{AC3} I. Assem and F. Coelho, Endomorphism rings of projective modules over laura algebras, $\emph{J.~Algebra~Appl.}~\bf{3}$ (2004), no. 1, 49--60.


\bibitem{ACT} I. Assem, F. U. Coelho and S. Trepode, The left and right parts of a module category, $\it{J.~Algebra}~\bf{281}$ (2004), No. 2, 518--534.


\bibitem{AS} I. Assem and A. Skowro$\acute{\text{n}}$ski, On some classes of simply connected algebras,   $\emph{Proc.~London~Math.~Soc.}~\bf{56}$ (1988), no. 3, 417--450


\bibitem{BG} O. Bretscher and P. Gabriel, The standard form of a representation-finite algebra, $\emph{Bull.~Soc.~Math.~France}~\bf{111}$ (1983), 21-40. 


\bibitem{BMR1} A. B. Buan, R. Marsh and I. Reiten, Cluster-tilted algebras, \emph{Trans. Amer. Math. Soc.}~
$\bf{359}$ (2007), no. 1, 323--332.


\bibitem{BMR2} A. Buan, R. Marsh, and I. Reiten, Cluster mutation via quiver representations, $\emph{Comment.~Math.~Helv.}~\bf{83}$ (2008), no.1, 143--177.


\bibitem{CL} F. Coelho and M. Lanzilotta, Algebras with small homological dimension, $\emph{Manuscripta~Math.}~\bf{100}$ (1999), 1-11.

\bibitem{CL2} F. Coelho and M. Lanzilotta, Weakly shod algebras, $\emph{J.~Algebra}~\bf{265}$ (2003), no. 1, 379--403.

\bibitem{GM} E. Green and E. Marcos, Convex subquivers and the finitistic dimension, $\emph{Illinois~J.~Math}~\bf{61}$ (2017), No. 3--4, 385--397.

\bibitem{H} D. Happel, Triangulated categories in the representation theory of finite dimensional algebras, $\emph{London~Mathematical~Society~Lecture~Note~Series}~\bf{119}$ (1988), Cambridge University Press, Cambridge.

\bibitem{HR} D. Happel and C. M. Ringel, Tilted algebras, $\it{Trans.~Amer.~Math.~Soc.}$ $\bf{274}$ (1982), no. 2, 399--443

\bibitem{HRS} D. Happel, I. Reiten, and S. O. Smal{\o}, Tilting in abelian categories and quasitilted algebras, $\emph{Mem. Amer. Math. Soc.}~\bf{575}$ (1996). 


\bibitem{RS} I. Reiten and A. Skowro$\acute{\text{n}}$ski, Characterizations of algebras with small homological dimensions, $\emph{Adv.~Math.}~\bf{179}$ (2003), no. 1, 122--154.

\bibitem{RS2} I. Reiten and A. Skowro$\acute{\text{n}}$ski, Generalized double tilted algebras, $\emph{J.~Math.~Soc.~Japan}~\bf{56}$ (2004), no. 1, 269-288.

\bibitem{S} A. Skowro$\acute{\text{n}}$ski, Simply connected algebras and Hochschild cohomology, In: $\emph{Representations~of~Algebras}$, Canad. Math. Soc. Conf. Proc., AMS, $\bf{14}$, (1993), 431-447.

\bibitem{SM} D. Smith, On tilting modules over cluster-tilted algebras, $\emph{Illinois~J.~Math.}~\bf{52}$ (2008), no. 4, 1223--1247.




\end{thebibliography}
\end{document}